\documentclass[leqno,11pt]{amsart} 
\usepackage{amssymb} 
\usepackage{mathrsfs} 
\usepackage[pdftex]{graphicx}
\usepackage{tikz} 
\usetikzlibrary{patterns, intersections, calc}  
\theoremstyle{plain} 
\newtheorem{theorem}{\indent\sc Theorem}[section]
\newtheorem{lemma}[theorem]{\indent\sc Lemma}
\newtheorem{corollary}[theorem]{\indent\sc Corollary}

\newtheorem{fact}[theorem]{\indent\sc Fact}  
\theoremstyle{definition} 
\newtheorem{definition}[theorem]{\indent\sc Definition}
\newtheorem{remark}[theorem]{\indent\sc Remark}

\def\C{{\mathbb{C}}}
\def\RC{{\widehat{{\mathbb{C}}}}}
\def\D{{\mathbb{D}}}
\def\H{{\mathbb{H}}}
\def\L{{\mathbb{L}}}
\def\M{{\mathscr{M}}}
\def\P{{\mathscr{P}}}
\def\Pi{{\mathbb{P}}}
\def\R{{\mathbb{R}}}
\def\Z{{\mathbb{Z}}}

\begin{document}

\title[Application of the Bloch--Ros principle]{Application of the Bloch--Ros principle \\ to ramification theorem} 

\author[S.~Kasao]{Shunsuke Kasao}


\renewcommand{\thefootnote}{\fnsymbol{footnote}}
\footnote[0]{2020\textit{Mathematics Subject Classification}.
Primary 53A10; Secondary 30D35, 30D45.}
\thanks{
This work was supported by JST SPRING Grant Nember JPMJSP2135. 
}

\address{
Division of Mathematical and Physical Sciences, Graduate School of Natural Science and Technology, Kanazawa University \endgraf
Kanazawa, 920-1192, \endgraf
Japan
}\email{kasao15039@gmail.com}


\maketitle

\begin{abstract}
The author and Kawakami revealed that the Picard little theorem, the Carath\'{e}odory--Montel theorem and the Fujimoto theorem---phenomena concerning omitted values in value distribution theory, normal family theory and the theory of Gauss maps of minimal surfaces, respectively---are not isolated results but can be discussed within a unified framework. 
We call this theoretical framework the Bloch--Ros principle. 
Furthermore, the value-distribution-theoretic properties of Gauss maps hold not only for minimal surfaces but also for other classes of surfaces that admit singularities, such as maxfaces in the Lorentz--Minkowski $3$-space and improper affine fronts in the affine $3$-space, thereby further extending this theoretical framework. 
In this paper, we provide a unified approach to phenomena concerning totally ramified values among value distribution theory, normal family theory and the theory of Gauss maps of these classes of surfaces. 
\end{abstract}


\section{Introduction} 
There exists the duality between normal family theory and value distribution theory of meromorphic functions, which is called the Bloch principle. 
Zalcman \cite{Za1975} formulated a more precise statement on it.
Moreover, inspired by this formulation, Ros \cite{RA} gave a similar formulation for minimal surface theory. 
Based on the Zalcman and Ros works, the author and Kawakami \cite{KK2025} restated the Zalcman lemma, which determines the normality for families of meromorphic functions (Fact \ref{thm:Zalcman}) and gave an effective criterion to determine which properties for meromorphic functions that play a role of Gauss maps of various classes of surfaces satisfy the Gaussian curvature estimate (Fact \ref{fact:m-curvature est}). 
As an application of these criterion, we also gave a systematic description of the relationship among the Liuoville theorem, the Montel theorem and generalizations of the Osserman theorem (the solution to the Nirenberg conjecture) as well as the Picard little theorem, the Carath\'{e}odory--Montel theorem and generalizations of the Fujimoto theorem. 
We call this theoretical framework the \textit{Bloch--Ros principle} in honor of the Bloch principle and the Ros work.

On the other hand, in the theory of value distribution of meromorphic functions, totally ramified values are studied as well as omitted values. 
Following \cite{KW2023,EMJ2026}, we recall the notion of a totally ramified value of a meromorphic function.

\begin{definition} \label{def:TRV}
Let $\mu \in \Z_{\geq 2} \cup \{ \infty \}$ and $g : \Sigma \longrightarrow \RC$ be a holomorphic map on a Riemann surface $\Sigma$. 
We say that $\alpha \in \RC$ is a \textit{totally ramified value} of $g$  of \textit{order} $\mu$ if the equation $ g = \alpha$ has no root of multiplicity less than $\mu$ on $\Sigma$. 
An omitted value $\alpha$ of $g$, i.e., a value not taken on $\Sigma$, is also regarded as a totally ramified value of $g$ of order $\infty$. 
For a totally ramified value of $g$ of order $\mu$, we also define its \textit{weight} as $1 - (1/\mu)$.
\end{definition}

For the Gauss map of a minimal surface in $\R^3$, Fujimoto \cite{Fu1992} proved that the Gaussian curvature estimate occurs if the sum of weights of totally ramified values, denoted by $\gamma$, is greater than $4$. 
In addition, Kawakami \cite{Ka2015} also gave a curvature estimate for the conformal metric $ds^2 =(1+|g|^2)^{m} \, |f|^2 \, |dz|^2$ on an open connected Riemann surface $\Sigma$, where $m$ is a positive integer, $f \, dz$ is a holomorphic $1$-form on $\Sigma$ and $g$ is a meromorphic function satisfying $\gamma > m + 2$ on $\Sigma$. 

In this paper, we reinterpret the above curvature estimate from the viewpoint of the Bloch--Ros principle. 
This approach reveals that the estimate is not an isolated phenomenon, but rather reflects a unified property underlying value distribution theory of meromorphic functions, normal family theory and the theory of Gauss maps of surfaces including minimal surfaces in $\R^3$, maxfaces in the Lorentz--Minkowski 3-space $\L^3$, improper affine fronts in the affine $3$-space $\R^3$ and flat fronts in the hyperbolic $3$-space $\H^3$. 

The paper is organized as follows: 
In Section \ref{sec:2}, we recall the duality between value distribution theory and normal family theory, and give a formulation to the theory of Gauss maps of various classes of surfaces. 
We first define the closedness and the compactness of a property in Definition \ref{def:cpt prop}.
Fact \ref{fact:merit of cpt}(A) and the Zalcman lemma (Fact \ref{thm:Zalcman}) represent the duality between value distribution theory and normal family theory. 
We define a Weierstrass $m$-triple $(\Sigma, \, f \, dz, \, g)$ in Definition \ref{def:m-pair} and $m$-curvature estimate for compact property in Definition \ref{def:m-curvature est}. 
Fact \ref{fact:m-curvature est} gives an effective criterion to determine which compact properties satisfy $m$-curvature estimate.
In Section \ref{sec:3}, we first define a certain property concerning totally ramified values in Theorem \ref{thm:Nev prop}. 
For simplicity, we refer to this property as ramification property. 
By Fact \ref{thm:Zalcman}, we prove that ramification property is a compact property in Theorem \ref{thm:Nev prop}. 
As a corollary of Fact \ref{fact:merit of cpt}(B), we also see in Corollary \ref{cor:Nev prop} that any holomorphic map satisfying $\gamma > 2$ on the punctured unit disk extends a holomorphic map on the unit disk. 
By applying Fact \ref{fact:m-curvature est}, we prove a ramification property satisfies $m$-curvature estimate in Theorem \ref{thm:main theorem}. 
Applications to the theory of Gauss maps of various classes of surfaces can be found in \cite[4 Applications]{Ka2015}. 
In conclusion, we can obtain the following trinity. 

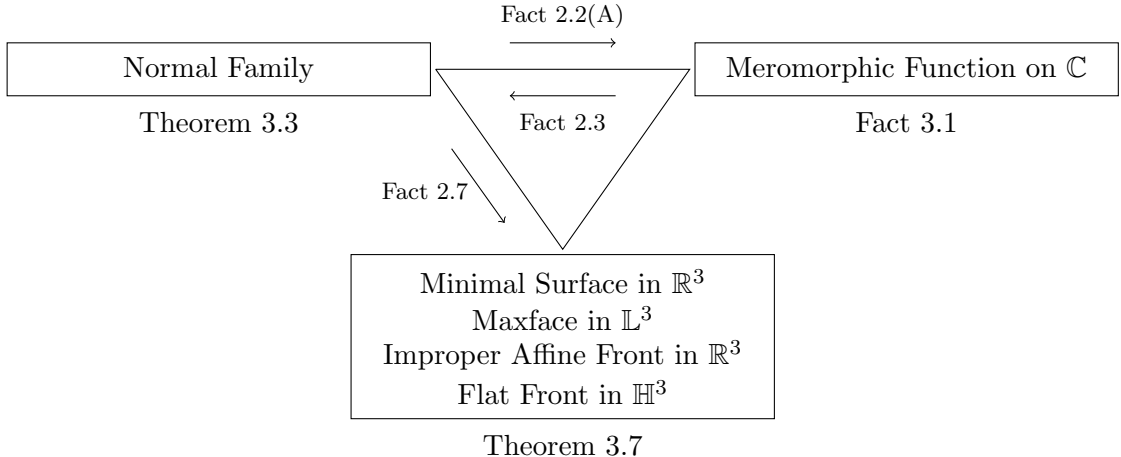
\vspace{5mm}
\begin{center}
\begin{tikzpicture}[scale=0.7]
\draw(-10.5,0) rectangle (-2.5,1);
\draw(2.5,0) rectangle (10.5,1);
\draw(-4,-6.1) rectangle (4,-3);
\node (A) at (-6.5,0.5) {Normal Family};
\node (A1) at (-6.5,-0.5) {Theorem \ref{thm:Nev prop}};
\node (B) at (6.5,0.5) {Meromorphic Function on $\mathbb{C}$};
\node (B1) at (6.5,-0.5) {Fact \ref{fact:Nevanlinna}};
\node (C1) at (0,-3.5) {Minimal Surface in $\mathbb{R}^3$};
\node (C2) at (0,-4.2) {Maxface in $\mathbb{L}^3$};
\node (C3) at (0,-4.9) {Improper Affine Front in $\mathbb{R}^3$};
\node (C4) at (0,-5.6) {Flat Front in $\mathbb{H}^3$};
\node (C5) at (0,-6.6) {Theorem \ref{thm:main theorem}};
\draw (-2.4,0.5)--(2.4,0.5)--(0,-2.9)--(-2.4,0.5);
\draw[->] (-1,1)--(1,1);
\node (D) at (0,1.5) {\footnotesize{Fact \ref{fact:merit of cpt}(A)}};
\draw[->] (1,0)--(-1,0);
\node (E) at (0,-0.5) {\footnotesize{Fact \ref{thm:Zalcman}}};
\draw[->] (-2.1,-1)--(-1.1,-2.4);
\node (G) at (-2.6,-1.8) {\footnotesize{Fact \ref{fact:m-curvature est}}};
\end{tikzpicture}
\end{center}

\section{Background and Preliminaries} \label{sec:2} 
\subsection{Bloch Principle} \label{sec:2.1}
In this part, we review the duality between normal family theory and value distribution theory of holomorphic maps.

We first define a half of the spherical distance $\chi(\cdot,\cdot)$ based on \cite[Section 1.3]{Fu1993}. 
The extended complex plane $\RC:=\C \cup \{ \infty \}$ is identified with the unit sphere $S^2(\subset \R^3)$ by the stereographic projection $\pi_N:\RC \longrightarrow S^2$ that maps $\infty \in \RC$ to $N=(0,0,1) \in S^2$.
Here, a distance on $S^2$ is the induced metric as a subset of $\R^3$. 
Then the distance $\chi(z_1,z_2)$ between $z_1$ and $z_2$ in $\RC$ is defined as a half of the distance on $S^2$. 
We call $\chi(z_1,z_2)$ a half of the \textit{spherical distance} between $z_1$ and $z_2$ or the \textit{chordal distance} between $z_1$ and $z_2$. For $z_1, \, z_2 \in \C$, we can express as follows:
\[
\chi(z_1, z_2) = \dfrac{|z_1-z_2|}{\sqrt{1+|z_1|^2} \,  \sqrt{1+|z_2|^2}}, \quad \chi(z_1, \infty) = \dfrac{1}{\sqrt{1+|z_1|^2}}. 
\]

Let $\Sigma$ be a connected Riemann surface. We define  
\begin{align*}
\M(\Sigma) 
& := \{ g : \Sigma \longrightarrow \RC \, : \, \mathrm{holomorphic \; map} \} \\
& = \{ g : \Sigma \longrightarrow \RC \, : \, \mathrm{meromorphic \; function} \} \cup \{ g \equiv \infty \; \mathrm{on} \; \Sigma \}.
\end{align*}
On $\M(\Sigma)$, we consider the topology of the uniform convergence on compact subsets.
We remark that a sequence $\{g_n\}_{n=1}^\infty$ converges to $g$ in $\M(\Sigma)$ in this topology if and only if $\{g_n\}_{n=1}^\infty$ converges locally uniformly to $g$ on $\Sigma$ with respect to the spherical distance $\chi$.    

Let $P$ be an arbitrary property for $\RC$-valued holomorphic maps and we set
\begin{align*}
\P(\Sigma):=\{ g \in \M(\Sigma) \;:\; g \;\mathrm{satisfies \; the \; property}\; P \}.
\end{align*}
For a property $P$, we give the following definitions based on \cite{RA}:

\begin{definition} \label{def:cpt prop}
Given a property $P$, we consider the following assertions: 
\begin{enumerate}
\item[(P$1$)] For any two Riemann surfaces $\Sigma$ and $\Sigma'$, and for any holomorphic map without ramification points $\phi:\Sigma \longrightarrow \Sigma'$, if $g \in \P(\Sigma')$, then $g\circ\phi \in \P(\Sigma)$. 
\item[(P$2$)] Let $\Sigma$ be a Riemann surface and $g \in \M(\Sigma)$. 
If we have $g|_{\Omega} \in \P(\Omega)$ for any relatively compact domain $\Omega$ of $\Sigma$, then $g \in \P(\Sigma)$.
\item[(P$3$)] For any Riemann surface $\Sigma$, the family $\P(\Sigma)$ is a closed subset of $\M(\Sigma)$.
\item[(P$4$)] For any Riemann surface $\Sigma$, the family $\P(\Sigma)$ is normal on $\Sigma$. 
\end{enumerate}
If $P$ satisfies the axioms (P$1$), (P$2$) and (P$3$), $P$ is called a \textit{closed property}. 
Furthermore, if $P$ satisfies the axioms (P$1$), (P$2$), (P$3$) and (P$4$), $P$ is called a \textit{compact property}.  
\end{definition}

We define $D(a ; R) := \{ z \in \C : |z-a| < R \}$ and $\D := D(0 ; 1)$.  
If $P$ is shown to be compact, the following properties are derived. 

\begin{fact}{\cite[Lemma 3]{RA}} \label{fact:merit of cpt}
Let $P$ be a compact property. 
Then the following holds:
\begin{enumerate}
  \item [(A)] $\P(\C)$ contains only constant maps.
  \item[(B)] If $g \in \P(\D \setminus \{ 0 \})$, then $z=0$ is not an essential singularity of $g$, that is, $g$ is a holomorphic map from $\D$ to $\RC$.
\end{enumerate}
\end{fact}


Let $\Omega$ be a domain in $\C$. 
Then we denote by $|\nabla g|_e$ the length of the Euclidean gradient of $\pi_N \circ g$ on $\Omega$. 
This corresponds to the spherical derivative in complex analysis (\cite[Section 1.3]{WB},\cite[Lemma 1.6]{CTC}\cite[Section 1.2]{Sc}). 
In fact, for a meromorphic function $g\colon \Omega \to \RC$, we obtain 
\begin{align*}
|\nabla g|_e^2=\dfrac{8\,|g'|^2}{(1+|g|^2)^2}, \; \mathrm{i.e.}, \; |\nabla g|_e=\dfrac{2\sqrt{2}\,|g'|}{1+|g|^2}.
\end{align*} 
and its spherical derivative is given by $g^{\#}=|g'|/(1+|g|^2)$. 
Note that we have $|\nabla g|_e \equiv 0$ on $\Omega$ if $g \equiv \infty$ on $\Omega$.  

The following theorem corresponds to the Zalcman lemma (for example, see \cite[3. The main lemma]{Za1975}) in complex analysis, which is useful in determining whether a property is compact. 

\begin{fact}{\cite[Lemma 3]{RA}}{\cite[Theorem 2.12]{KK2025}} \label{thm:Zalcman}
Let $P$ be a closed property. 
Then $\rm{(\hspace{.18em}i\hspace{.18em})}$ or $\rm{(\hspace{.08em}ii\hspace{.08em})}$ holds: 
\begin{enumerate}
\item[\rm(\hspace{.18em}i\hspace{.18em})] $P$ is a compact property. 
\item[\rm(\hspace{.08em}ii\hspace{.08em})] There exists $f \in \P(\C)$ such that $|\nabla f|_e(0)=1$ and $|\nabla f|_e \leq 1$ on $\C$. 
\end{enumerate}
\end{fact}

\subsection{Bloch--Ros Principle} \label{sec:2.2}

In this part, we survey the trinity among value distribution theory of holomorphic maps, normal family theory and the theory of Gauss maps of several classes of surfaces including minimal surfaces in $\R^3$, maxfaces in $\L^3$, improper affine fronts in $\R^3$ and flat fronts in $\H^3$. 

We first introduce the notion of Weierstrass $m$-triples. 
This is a generalization of Weierstrass data for minimal surfaces in $\R^{3}$. 

\begin{definition} \label{def:m-pair}
Let $\Sigma$ be an open Riemann surface with the conformal metric
\begin{align} \label{eq:m-metric}
ds^2 := (1+|g|^2)^m \, |f|^2 \, |dz|^2,
\end{align}
where $f \, dz$ is a holomorphic $1$-form on $\Sigma$, $g$ is a meromorphic function on $\Sigma$ and $m \in \Z_{>0}$. Then we call this triple $(\Sigma, \, f \, dz, \, g)$ a \textit{Weierstrass $m$-triple}.   
\end{definition}

\begin{remark}
We note that
\begin{align} \label{eq:regularity}
& 0 < (1+|g|^2)^m \, |f|^2 < +\infty \quad \Longleftrightarrow \quad (f \, dz)_0 = m \, (g)_\infty,  
\end{align}
where $(f \, dz)_{0}$ is the zero divisor of $f \, dz$ and $(g)_{\infty}$ is the polar divisor of $g$. 
Moreover, for a conformal metric $ds^{2}= (\lambda(z))^2 \, |dz|^2$, we have 
$K_{ds^2} =-(\Delta\log{\lambda})/ \lambda^{2}$. 
Hence the Gaussian curvature $K_{ds^2}$ of any Weierstrass $m$-triple $(\Sigma, \, f \, dz, \, g)$ is given by
\begin{align*} 
K_{ds^2} = - \dfrac{2 \, m \, |g'|^2}{(1+|g|^2)^{m+2} \, |f|^2}.
\end{align*}
\end{remark}

\begin{definition} \label{def:m-curvature est}
Let $P$ be a compact property for $\RC$-valued holomorphic maps and $m \in \Z_{>0}$. 
We say that $P$ satisfies \textit{$m$-curvature estimate} if there exists a positive constant $C=C(P,m)>0$ such that for any Weierstrass $m$-triple $(\Sigma, \, f \, dz, \, g)$ satisfying $g \in \P(\Sigma)$, the following inequality holds: 
\begin{align*}
|K_{ds^2}|^{\frac{1}{2}}  \leq \dfrac{C}{d} \quad\mathrm{on} \; \Sigma.
\end{align*}
Here, $K_{ds^2}$ is the Gaussian curvature with respect to the metric $ds^2$ given by \eqref{eq:m-metric} and $d$ is its geodesic distance to the boundary of $\Sigma$.
\end{definition}

The following result is an effective criterion to determine which compact properties satisfy $m$-curvature estimate.

\begin{fact}{\cite[Theorem3.4]{KK2025}} \label{fact:m-curvature est}
Let $P$ be a compact property for $\RC$-valued holomorphic maps and $m \in \Z_{>0}$. 
Then $\rm(\hspace{.18em}i\hspace{.18em})$ or $\rm(\hspace{.08em}ii\hspace{.08em})$ holds:

\begin{enumerate}
  \item[\rm(\hspace{.18em}i\hspace{.18em})] $P$ satisfies $m$-curvature estimate. 
  \item[\rm(\hspace{.08em}ii\hspace{.08em})] There exists a Weierstrass $m$-triple $(\D, \, f \, dz, \, g)$ such that 
\begin{enumerate}
  \item[(A)] nonconstant meromorphic function $g \in \P(\D)$,
  \item[(B)] $ds^2 :=(1+|g|^2)^m \, |f|^2 \, |dz|^2$ is complete on $\D$, 
  \item[(C)] $|K_{ds^2}(0)|=\dfrac{1}{4}$, and $|K_{ds^2}| \leq 1$ on $\D$. 
\end{enumerate}  
\end{enumerate}
 \end{fact}

\section{Main Results} \label{sec:3}
\subsection{Compactness of Ramification Property} \label{sec:3.1}

As for the results concerning totally ramified values, the following is widely recognized. 

\begin{fact}{\cite{CTC},\cite{Ne1970},\cite{Ru2021}} \label{fact:Nevanlinna}
For $q \in \Z_{>0}$, let $\alpha(q) := \{ \alpha_1, \, \cdots, \, \alpha_q  \} \subset \RC$ be a set of distinct points and let $\mu_1, \, \cdots, \, \mu_q \in \Z_{\geq 2} \cup \{ \infty \}$ satisfy $\gamma := \sum_{j=1}^q (1 - (1/\mu_j) ) > 2$.
Then any meromorphic on $\C$ satisfying the condition that each $\alpha_j$ is a totally ramified value of order $\mu_j$ is constant. 
\end{fact}

\remark \label{rem:Picard and Nevanlinna}
For a nonconstant meromorphic function $g$ on $\C$, let $D_g$ denote the number of omitted values of $g$. 
By applying Fact \ref{fact:Nevanlinna} to $\{ \alpha_1, \cdots, \alpha_{D_g} \} = \RC \setminus g(\C)$,  
\begin{align*}
  D_g = \gamma \leq 2, 
\end{align*} 
that is, the Picard little theorem holds. 

\begin{theorem} \label{thm:Nev prop}
For $q \in \Z_{>0}$, let $\alpha(q) := \{ \alpha_1, \, \cdots, \, \alpha_q  \} \subset \RC$ be a set of distinct points and let $\mu_1, \, \cdots, \, \mu_q \in \Z_{\geq 2} \cup \{ \infty \}$. 
For a Riemann surface $\Sigma$ and a holomorphic map $g:\Sigma \longrightarrow \RC$, we define the property $P_{\alpha(q)}$ as 
\begin{center}
each $\alpha_j$ is a totally ramified value of $g$ of order at least $\mu_j$ \; or \; $g \equiv $ $($constant$)$ on $\Sigma$. 
\end{center}
Then $P_{\alpha(q)}$ is a closed property. 
Furthermore, $P_{\alpha(q)}$ is a compact property if $\gamma := \sum_{j=1}^q (1 - (1/\mu_j) ) > 2$.
\end{theorem}

\begin{proof}
The case $\mu_j = \infty$ implies that $\alpha_j$ is an omitted value, which has already been proved in \cite[Theorem 2.14]{KK2025}. 
Hence, we focus on the case $\mu_j \in \Z_{\geq 2}$ in the following. \\
(P$1$) Let $\Sigma$ and $\Sigma'$ be two Riemann surfaces, let $\phi : \Sigma \longrightarrow \Sigma'$ be a holomorphic map without ramification points and $g \in \P_{\alpha(q)}(\Sigma')$. 
If each $\alpha_j$ is a totally ramified value of $g$ of order at least $\mu_j$, 
\begin{enumerate}
  \item[] $(g \circ \phi)^{-1} (\alpha_j) \ne \emptyset \Longrightarrow \mathrm{ord}_{(g \circ \phi) - \alpha_j}(p) \geq \mu_j \;\, (^{\forall} p \in (g \circ \phi)^{-1} (\alpha_j))$, 
  \item[] $(g \circ \phi)^{-1} (\alpha_j) = \emptyset \Longrightarrow$ $\alpha_j$ is an omitted value, i.e., a totally ramified value of order $\infty$. 
\end{enumerate}
Therefore, $g \circ \phi \in \P_{\alpha(q)}(\Sigma)$. 
If $g \equiv$ (constant) on $\Sigma'$, $g \circ \phi \equiv$ (constant) on $\Sigma$, that is, $g \circ \phi \in \P_{\alpha(q)}(\Sigma)$. \\
(P$2$) It is easy to see that $P_{\alpha(q)}$ satisfies this condition. \\
(P$3$) For a Riemann surface $\Sigma$, we consider any sequence $\{ g_n \}_{n=1}^\infty$ in $\P_{\alpha(q)}(\Sigma)$ satisfying $\{ g_n \}_{n=1}^\infty$ converges locally uniformly to $g$ on $\Sigma$ with respect to the spherical distance $\chi$. 
Assume that the limit function $g$ is nonconstant on $\Sigma$. 
Then we can prove that each $\alpha_j$ is a totally ramified value of $g$ of order at least $\mu_j$ by contradiction. 
Suppose that there exists $p \in \Sigma$ such that $\mathrm{ord}_{g - \alpha_j}(p) < \mu_j$. 
By the Hurwitz theorem, there exists a neighborhood $\Omega_p$ of $p$ and $N \in \Z_{>0}$ such that for any $n \geq N$, 
\begin{align*}
  \mu_j 
  > \mathrm{ord}_{g - \alpha_j}(p)
& = ( \mathrm{the \; order \; of \; zero \; points \; of \;} g - \alpha_j \; \mathrm{in} \; \Omega_p ) \\ 
& = ( \mathrm{the \; order \; of \; zero \; points \; of \;} g_n - \alpha_j \; \mathrm{in} \; \Omega_p ).
\end{align*}
Therefore, $^{\forall} n \geq N, \, ^{\exists} p_n \in \Omega_p$ s.t. $g_n(p_n) - \alpha_j = 0$ and $\mathrm{ord}_{g_n - \alpha_j}(p_n) < \mu_j$.  
From $g_n \in \P_{\alpha(q)}(\Sigma)\;(n \geq N)$, each $g_n\;(n \geq N)$ is  constant on $\Sigma$. 
Since $\{ g_n \}_{n=1}^\infty$ converges locally uniformly to $g$ on $\Sigma$ with respect to $\chi$, $g$ is constant on $\Sigma$. 
This is a contradiction. \\
(P$4$) We next show closed property $P_{\alpha_(q)}$ is compact if $\gamma > 2$. 
Suppose that $P_{\alpha(q)}$ is not compact. 
By Fact \ref{thm:Zalcman}, there exists $g \in \P_{\alpha(q)}(\C)$ satisfying $|\nabla g|_e(0)=1$. 
Note that $g$ is a nonconstant meromorphic function on $\C$. 
So each $\alpha_j$ is a totally ramified value of $g$ of order at least $\mu_j$.  
By Fact \ref{fact:Nevanlinna}, $g$ is constant on $\C$. 
This is a contradiction.  
\end{proof}

By applying Fact \ref{fact:merit of cpt}(B) to $\P_{\alpha(q)}$, we obtain the following.

\begin{corollary} \label{cor:Nev prop}
For $q \in \Z_{>0}$, let $\alpha(q) := \{ \alpha_1, \, \cdots, \, \alpha_q  \} \subset \RC$ be a set of distinct points and let $\mu_1, \, \cdots, \, \mu_q \in \Z_{\geq 2} \cup \{ \infty \}$ satisfy $\gamma := \sum_{j=1}^q (1 - (1/\mu_j) ) > 2$.
If a holomorphic map $g : \D \setminus \{ 0 \} \longrightarrow \RC$ satisfies the condition that each $\alpha_j$ is a totally ramified value of $g$ of order $\mu_j$, the puncture $z = 0$ is not an essential singularity of $g$.
\end{corollary}

\subsection{Curvature Estimate of Ramification Property} \label{sec:3.2}

We verify that compact property $P_{\alpha(q)}$ satisfies $m$-curvature estimate if the sum $\gamma$ of weights is greater than $m+2$, which corresponds to \cite[Theorem 2.1]{Ka2015}. 

We first recall the following function-theoretic lemma. 

\begin{fact}\cite[Corollary 1.4.15]{Fu1993} \label{fac:Fujimoto lemma}
For $q \in \Z_{>0}$, let $\alpha(q) := \{ \alpha_1, \, \cdots, \, \alpha_q  \} \subset \RC$ be a set of distinct points and let $\mu_1, \, \cdots, \, \mu_q \in \Z_{\geq 2} \cup \{ \infty \}$ satisfy $\gamma := \sum_{j=1}^q (1 - (1/\mu_j) ) > 2$.
If a holomorphic map $g : D(0 ; R) \longrightarrow \RC$ satisfies the condition that each $\alpha_j$ is a totally ramified value of $g$ of order $\mu_j$, then, for arbitrary constants $\eta \geq 0$ and $\delta > 0$ with $\gamma - 2 > \gamma \, (\eta + \delta)$, there exists a positive constant $C'$, depending only on $\gamma$, $\eta$, $\delta$ and $L := \min_{i < j} \chi(\alpha_i, \alpha_j)$, such that 
\begin{align*}
\dfrac{|g'|}{(1+|g|^2) \, \prod_{j=1}^q (\chi(g,\alpha_j)^{1-1/\mu_j})^{1-(\eta + \delta)}} \leq C' \, \dfrac{R}{R^2-|z|^2} \quad \mathrm{on} \; D(0;R).
\end{align*}
\end{fact}

We also need the following inequality, which is derived from the Yau--Schwarz lemma (\cite[Theorem 2]{Yau}).

\begin{fact}\cite[Lemma 3.8]{KK2025} \label{fact:Yau-Schwarz}
Let $ds^2 = (\lambda(z))^2 \, |dz|^2$ be a complete Hermitian metric on the unit disk $\D$ whose Gaussian curvature $K_{ds^2}$ satisfies $K_{ds^2} \geq -1$. Then 
\begin{align*}
\lambda(z) \geq \dfrac{2}{1-|z|^2} \quad\mathrm{on}\;\D.
\end{align*}
\end{fact}

The main result of this paper is the following. 

\begin{theorem} \label{thm:main theorem}
For $q \in \Z_{>0}$, let $\alpha(q) := \{ \alpha_1, \, \cdots, \, \alpha_q  \} \subset \RC$ be a set of distinct points, let $\mu_1, \, \cdots, \, \mu_q \in \Z_{\geq 2} \cup \{ \infty \}$ and let $P_{\alpha(q)}$ be the property in Theorem \ref{thm:Nev prop}. 
If 
\begin{align*}
  \gamma := \sum_{j=1}^q \left(1 - \dfrac{1}{\mu_j} \right) > m + 2,
\end{align*} 
then $P_{\alpha(q)}$ is a compact property that satisfies $m$-curvature estimate. 
In particular, provided that $\gamma > m + 2$, there exists a positive constant $C > 0$ such that for any Weierstrass $m$-triple $(\Sigma, \, f \, dz, \, g)$ satisfying the condition that each $\alpha_j$ is a totally ramified value of $g$ of order $\mu_j$, the following inequality holds:
\begin{align*}
|K_{ds^2}(p)|^{\frac{1}{2}} \leq \dfrac{C}{d(p)} \quad(^\forall p \in \Sigma).
\end{align*}
\end{theorem}

\begin{proof}
In the case where $\gamma > m + 2$, we prove that compact property $P_{\alpha(q)}$ satisfies $m$-curvature estimate by contradiction. 
Suppose that compact property $P_{\alpha(q)}$ does not satisfy $m$-curvature estimate. 
By Fact \ref{fact:m-curvature est}, there exist a Weierstrass $m$-triple $(\D, \, f \, dz, \, g)$ such that (A), (B) and (C) in Fact \ref{fact:m-curvature est} hold. 
Here, we may assume that $\alpha_{m+3}=\infty$ by composing with a suitable M\"{o}bius transformation if necessary. 

We choose a constant $\delta > 0$ such that $\gamma - (m+2) > 2 \, \gamma \, \delta >0$ and we set
\begin{align*}
  \eta := \dfrac{\gamma - (m+2) - 2 \, \gamma \, \delta}{\gamma} (>0), \quad \eta_j := 1 - \dfrac{1}{\mu_j} \left( \in \left[ \dfrac{1}{2}, 1 \right] \right).
\end{align*}
We remark that 
\begin{align*}
  \gamma \, (\eta + \delta) = \gamma - (m+2) - \gamma \, \delta < \gamma - 2. 
\end{align*}
By Fact \ref{fac:Fujimoto lemma}, there exists $C' > 0$ such that  
\begin{align} \label{eq:first Fujimoto}
\dfrac{|g'|}{(1+|g|^2) \, \prod_{j=1}^q (\chi(g, \alpha_j)^{\eta_j})^{1-(\eta + \delta)}} \leq  \dfrac{C'}{1-|z|^2} \quad \mathrm{on}\;\D.
\end{align}
From $\chi(g,\alpha_j) \leq |g-\alpha_j| \, (1+|g|^2)^{-\frac{1}{2}} \; (1 \leq j \leq q-1)$ and $\chi(g,\alpha_q) = (1+|g|^2)^{-\frac{1}{2}}$,
\begin{align*}
(1+|g|^2)  \prod_{j=1}^q (\chi(g, \alpha_j)^{\eta_j})^{1-(\eta + \delta)} 
& \leq (1+|g|^2)  \left( \prod_{j=1}^{q-1}  |g - \alpha_j|^{\eta_j} \right)^{1-(\eta + \delta)}  \left( \prod_{j=1}^q (1+|g|^2)^{-\frac{1}{2} \eta_j} \right)^{1-(\eta + \delta)} \\
& = \left( \prod_{j=1}^{q-1}  |g - \alpha_j|^{\eta_j} \right)^{1-(\eta + \delta)}  (1 + |g|^2)^{1 - \frac{1}{2} \gamma \{1 - (\eta + \delta)\} }. 
\end{align*}
Combining \eqref{eq:first Fujimoto} with this inequality, we obtain
\begin{align} \label{eq:second Fujimoto}
\dfrac{|g'| \, (1+|g|^2)^{\frac{1}{2} \gamma \{ 1-(\eta+\delta) \} -  1}}{\left(\prod_{j=1}^{q-1}|g-\alpha_j|^{\eta_j} \right)^{1 - (\eta + \delta) }} 
\leq \dfrac{C'}{1-|z|^2} \quad \mathrm{on}\;\D. 
\end{align}
We set $r := \dfrac{1}{2} \gamma \{ 1-(\eta+\delta) \} -  1$. 
Note that 
\begin{align*}
  r = \dfrac{1}{2} \{ \gamma - \gamma \, (\eta + \delta) -2 \} = \dfrac{1}{2} \, (m + 2 + \gamma \, \delta - 2 ) = \dfrac{1}{2} \, (m + \gamma \, \delta) > \dfrac{m}{2} \, ( > 0).
\end{align*}
By applying Fact \ref{fact:Yau-Schwarz} to the metric \eqref{eq:m-metric}, we have
\begin{align} \label{eq:Yau}
\dfrac{1}{1-|z|^2} < (1+|g|^2)^{\frac{m}{2}} \, |f| \, , \quad \mathrm{i.e.,} \quad \left( \frac{1}{1-|z|^2} \right)^{\frac{2r}{m}} < (1+|g|^2)^r \, |f|^{\frac{2r}{m}} \quad \mathrm{on} \; \D.
\end{align}
Here, we verify the following result:

\begin{lemma} \label{lem:order of h}
The function 
\begin{align*}
  h := \dfrac{|g'|}{\left( \prod_{j=1}^{q-1} |g-\alpha_j|^{\eta_j} \right)^{1-(\eta+\delta)}}
\end{align*}
vanishes at each point of $g^{-1}(\alpha_j)$.   
\end{lemma}
\begin{proof}[\textit{Proof of Lemma \ref{lem:order of h}}]
Set 
\begin{align*}
  \nu := \eta + \delta \, (\in (0, \, 1)), \quad \rho := \dfrac{\gamma - 2 - \gamma \, \nu}{2 \, (1 - \nu)} \, (>0). 
\end{align*}
Consider first a point $z_0 \in \D$ with $g(z_0) = \alpha_j$ for some $j=1, \, \cdots, \, q-1$, and let $\mu \, (\geq \mu_j)$ be the order of the zero of $g-\alpha_j$ at $z_0$.
Then, since g' has a zero of order $\mu - 1$ at $z_0$, the order of $h$ at $z_0$ is 
\begin{align*}
      \mu - 1 - \mu \eta_j (1 - \nu) 
  & = \mu - 1 - \left(\mu - \dfrac{\mu}{\mu_j} \right) + \mu \left( 1 - \dfrac{1}{\mu_j} \right) \nu \\
  & = \dfrac{\mu}{\mu_j} - 1 + \mu \left( 1 - \dfrac{1}{\mu_j} \right) \nu > 0.
\end{align*}
Next, assume that $g$ has a pole of order $\mu \, (\geq \mu_q)$ at $z_0 \in \D$. 
Then, since $g'$ has a pole of order $\mu + 1$ at $z_0$, the order of $h$ at $z_0$ is 
\begin{align*}
      - \mu - 1 + ( \eta_1 + \cdots + \eta_{q-1} ) \mu (1 - \nu )
  & = - \mu - 1 + \gamma \mu (1-\nu) - \eta_q \mu (1-\nu) \\
  & > \mu \gamma (1 - \nu) - \mu - 1 - \eta_q \mu \\
  & = \mu (2 + 2 \rho (1 - \nu)) - \mu -1 - \left( \mu - \dfrac{\mu}{\mu_q} \right) \\
  & = 2 \mu \rho (1 - \nu) + \left( \dfrac{\mu}{\mu_q} - 1 \right) \\
  & \geq 2 \mu \rho (1 - \nu) \, ( > 0).
\end{align*}
In particular, the order of $h$ is greater than $2 \mu \rho (1 - \nu) ( > 0)$.
\end{proof}

By Lemma \ref{lem:order of h}, \eqref{eq:second Fujimoto} and \eqref{eq:Yau}, we obtain 
\begin{align*}
         \left( \frac{1}{1-|z|^2} \right)^{\frac{2r}{m}} \, \dfrac{|g'|}{\left( \prod_{j=1}^{q-1} |g-\alpha_j|^{\eta_j} \right)^{1-(\eta+\delta)}} 
  & \leq \dfrac{|g'| \, (1+|g|^2)^r \, |f|^{\frac{2r}{m}}}{\left(\prod_{j=1}^{q-1}|g-\alpha_j|^{\eta_j} \right)^{1 - (\eta + \delta) }}  \leq \dfrac{C'}{1 - |z|^2} \, |f|^{\frac{2r}{m}} \quad \mathrm{on} \; \D.
\end{align*}
On $D':= \D \setminus \left( (g')^{-1}(0) \cup \bigcup_{j=1}^q g^{-1}(\alpha_j) \right)$, we thereby obtain 
\begin{align*}
\dfrac{1}{C'} \, \dfrac{1}{(1-|z|^2)^{ \frac{2r}{m} - 1 }} 
\leq \dfrac{|f|^{ \frac{2r}{m} } \, \left( \prod_{j=1}^{q-1}|g-\alpha_j|^{\eta_j} \right)^{ 1 - (\eta + \delta) }}{|g'|}.
\end{align*}
Since $r' := \dfrac{1}{(2r/m) - 1} > 0$, we have
\begin{align} \label{eq:final Fujimoto}
\dfrac{(C')^{-r'}}{1-|z|^2} \leq \left( \dfrac{|f|^{ \frac{2r}{m} } \, \left( \prod_{j=1}^{q-1}|g-\alpha_j|^{\eta_j} \right)^{ 1 - (\eta + \delta) }}{|g'|} \right)^{r'} \quad \mathrm{on} \; D'.
\end{align}
Let $\lambda$ be the right hand side of $\eqref{eq:final Fujimoto}$ and we consider the conformal metric $ds_0^2 := \lambda^2 \, |dz|^2$ on $D'$. 
By applying $\Delta \log|F| = 0$ for any holomorphic function $F$ without zeros, we obtain $K_{ds_0^2}=-\Delta \log \lambda/\lambda^2 \equiv 0$ on $D'$, that is, $ds_0^2$ is a flat metric on $D'$. 
Furthermore, $ds_0^2$ is a complete metric on $D'$. 
Indeed, the left hand side of \eqref{eq:final Fujimoto} is the Poincar\'{e} metric and, by Lemma \ref{lem:order of h} and \eqref{eq:regularity}, $\lambda$ diverges at each puncture of $D'$.
Therefore the universal covering of $(D',ds_0^2)$ is $\C$. 
This is a contradiction. 
\end{proof}

As a corollary of Theorem \ref{thm:main theorem}, we give the following ramification theorem (\cite[Corollary 2.2]{Ka2015}).

\begin{corollary}\label{cor:Kawakami2015}
For $q \in \Z_{>0}$, let $\alpha(q) := \{ \alpha_1, \, \cdots, \, \alpha_q  \} \subset \RC$ be a set of distinct points, let $\mu_1, \, \cdots, \, \mu_q \in \Z_{\geq 2} \cup \{ \infty \}$ satisfy $\gamma := \sum_{j=1}^q (1 - (1/\mu_j) ) > m + 2$ and let $(\Sigma, \, f \, dz, \, g)$ be a Weierstrass $m$-triple.
If the metric given by \eqref{eq:m-metric} is complete on $\Sigma$ and each $\alpha_j$ is a totally ramified value of $g$ of order $\mu_j$, 
then $g$ is constant on $\Sigma$.
\end{corollary}
\begin{proof}
By Theorem \ref{thm:main theorem}, there exists $C > 0$ such that $|K_{ds^2}(p)|^{\frac{1}{2}} \leq C/d(p) \;(^\forall p \in \Sigma)$. 
Since $d(p) = \infty$ for any $p \in \Sigma$ in the case where $ds^2$ is a complete metric on $\Sigma$, $K_{ds^2} \equiv 0$ on $\Sigma$, that is, $g$ is constant on $\Sigma$.
\end{proof}

Finally, we also obtain generalizations of the Fujimoto theorem. 

\begin{corollary} \cite[Corollary 3.11]{KK2025} \label{cor:KK2025}
Let $(\Sigma, \, f \, dz, \, g)$ be a Weierstrass $m$-triple.
If the metric given by \eqref{eq:m-metric} is complete on $\Sigma$ and $g$ is nonconstant on $\Sigma$, then $g$ can omit at most $m + 2$ distinct values. 
Furthermore, this result is optimal.
\end{corollary}

\begin{proof}
Let $D_g$ denote the number of omitted values of $g$. 
By applying Corollary \ref{cor:Kawakami2015} to $\{ \alpha_1, \cdots, \alpha_{D_g} \} = \RC \setminus g(\Sigma)$, we obtain
\begin{align*}
  D_g =\gamma \leq m+2. 
\end{align*}
\end{proof}

\textbf{Acknowledgements.} The author would like to sincerely express to Yu Kawakami for invaluable advice and guidance throughout this research. The author also thanks Kazuya Tohge, Atsuhira Nagano, Tomoki Kawahira for their warm encouragement. This work was supported by JST SPRING, Grant Nember JPMJSP2135.


\end{document}